\documentclass{amsart}

\usepackage{amssymb,amsmath,mathtools,xcolor,graphicx,xspace,colortbl,ragged2e,rotating} %
\usepackage{amsfonts}  
\usepackage{amsmath}  
\usepackage{amssymb}  
\usepackage{amssymb,amsmath,mathtools,xcolor,graphicx,xspace,colortbl,ragged2e,rotating}  
\usepackage[all]{xy}  
\usepackage[utf8]{inputenc}
\graphicspath{{Final_Hypersimple Rings_graphics/}{Final_Hypersimple Rings_tcache/}{Final_Hypersimple Rings_gcache/}}
\DeclareGraphicsExtensions{.pdf,.eps,.ps,.png,.jpg,.jpeg}
\graphicspath{{Hypersimple_Rings_graphics/}{Hypersimple_Rings_tcache/}{Hypersimple_Rings_gcache/}}
\DeclareGraphicsExtensions{.pdf,.eps,.ps,.png,.jpg,.jpeg}
\graphicspath{{Hypersimple_Rings_graphics/}{Hypersimple_Rings_tcache/}{Hypersimple_Rings_gcache/}}
\DeclareGraphicsExtensions{.pdf,.eps,.ps,.png,.jpg,.jpeg}
\newtheorem{theorem}{Theorem}[section]
\theoremstyle{plain}

\newtheorem{corollary}[theorem]{Corollary}

\newtheorem{definition}[theorem]{Definition}
\newtheorem{example}[theorem]{Example}

\newtheorem{lemma}[theorem]{Lemma}

\newtheorem{problem}[theorem]{Problem}
\newtheorem{proposition}[theorem]{Proposition}
\newtheorem{remark}[theorem]{Remark}

\numberwithin{equation}{section}

\input{tcilatex}
\begin{document}
\title{ {Hypersimple Rings and Modules}}
\author[Lomp]{Christian Lomp}
\address{CMUP, Departamento de Matemática, Faculdade de Ciências, Universidade do Porto, rua do Campo Alegre s/n, 4169-007 Porto, Portugal.}
\email{clomp@fc.up.pt}
\author[Yousif]{Mohamed Yousif}
\address{Department of Mathematics, The Ohio State University, Lima, Ohio 45804, USA.}
\email{yousif.1@osu.edu}
\author[Zhou]{Yiqiang Zhou}
\address{Department of Mathematics and Statistics, Memorial University of Newfoundland, St. John's, NL,
A1C 5S7, Canada.}
\email{zhou@mun.ca}
\date{
\today
}
\dedicatory{Dedicated to Andr{\'e} Leroy on his retirement}
\begin{abstract}In this paper a simple right $R$-module $S$ over a ring $R$ is called hypersimple if its injective hull $E (S)$ is cyclic, and a ring $R$ is called right hypersimple if every simple right $R$-module is hypersimple. We intiate a study of these new notions, and revisit Osofsky's work on hypercyclic rings, i.e. rings whose cyclic right modules have cyclic injective hulls. 
\end{abstract}
\maketitle

\section{Introduction}
The Pr{\"u}fer groups are the injective hulls of the simple Abelian groups and they are Artinian, but not Noetherian.
More generally any injective hull of a simple module over a commutative Noetherian ring is Artinian as it was shown by Matlis in his seminal work \cite{M}.
Other finiteness conditions on the injective hull of modules have been considered for example by Rosenberg and Zelinsky in \cite{RZ} and Faith in \cite{Faith}. Faith and Walker for example proved in \cite[Theorem 5.5]{FW}
that a ring $R$ is quasi-Frobenius if and only if any injective right $R$-module is a direct sum of cyclic modules which are isomorphic to principal indecomposable right ideals of $R$. Furthermore, by Zorn's lemma one can easily prove that every injective module is the injective hull of a direct sum of cyclic modules.
In \cite{O2}, Osofsky studied the rings whose cyclic modules are injective
and showed that such rings are precisely the semisimple Artinian ones. Inspired by this result, Caldwell in \cite{C}
studied a class of rings, called (right)  {hypercyclic rings}, whose cyclic right modules have cyclic injective hulls. He proved that
a left perfect, right hypercyclic ring is Artinian and uniserial. Hypercyclic rings were thoroughly investigated by Caldwell for commutative rings in \cite{C}, and by Osofsky for noncommutative rings in \cite{O1}.
However, the only example of a hypercyclic ring that is not semisimple Artinian was provided by Caldwell in \cite{C},
and such a ring is commutative and self-injective. Furthermore, Caldwell has asked in his thesis \cite[Page 53]{C 1}
whether every hypercyclic ring is self-injective, and conjectured yes as an answer to his question. The conjecture still remains open, among several other
questions on the subject. For example it is not known if the notion of hypercyclic rings is left-right symmetric. Moreover, Osofsky asked in \cite{O1}
if the (Jacobson) radical of local hypercyclic rings is nil. 

Motivated by Caldwell's conjecture, it was shown in \cite{IY 5}
that if $R$ is a ring such that $E (R_{R})$ is cyclic and Dedekind-finite, then $R$ is right self-injective. In particular, if $E (R_{R})$ is cyclic then $R$ is right self-injective in the cases where $R$ is commutative, finite-dimensional, semilocal, or strongly $\pi $-regular. 

\bigskip We should point out that there is a gap in the proof of one of the
main results in \cite[Proposition 1.6]{O1},
where it was claimed that if $R$ is a semilocal right hypercyclic ring, then $s o c (R_{R})  \subseteq ^{e s s}R_{R}$. Subsequently, Lemma 1.12, Lemma 1.15 and Theorem 1.18 of \cite{O1}
remain uncertain. Moreover, because of the gap in \cite[Proposition 1.6]{O1},
\cite[Theorem 3.4]{JS}
is uncertain. 

\bigskip In this paper a simple right $R$-module $S$ is called hypersimple if $E (S)$ is cyclic, and a ring $R$ is called right hypersimple if $E (S)$ is cyclic for every simple right $R$-module $S$. We initiate a study of hypersimple rings and modules, and show that some of the work on hypercyclic rings can be obtained by requiring
only $E (R_{R})$ and/or $E (R/J (R))$ to be cyclic. 

\bigskip While we are unable to fix the gap in \cite[Proposition 1.6]{O1},
we rectify the situation by requiring the ring $R$ to be essentially right duo; i.e. every essential right ideal of $R$ is two-sided. Under this extra mild condition, Osofsky's result follows immediately by mimicking her argument, and some of her other results can be streamlined.  For example, we prove in Theorem \ref{If R is semilocal ess rt duo with E(R) and E(R/J(R)) are cyclic, then R is rt PF-ring}, that if
$R$ is a semilocal essentially right duo ring such that $E (R_{R})$ and $E (R/J (R))$ are cyclic modules, then $R$ is right pseudo-Frobenius whose right ideals are quasi-injective. 

\bigskip Throughout,
all rings $R$ are associative with unity and all modules are unitary. For a module $M$, we use $r a d (M)$, $s o c (M)$, $E (M)$ and $E n d (M)$ to denote the Jacobson radical, the socle, the injective hull and the endomorphism ring of $M$, respectively$\text{.}$ For a ring $R$, $r a d (_{R} R) =r a d (R_{R})$, which is the Jacobson radical of $R$ and is denoted by $J (R)$ or $J$. In general, $s o c (_{R} R) \neq s o c (R_{R})$, but we simply use $s o c (R)$ when $s o c (_{R} R) =s o c (R_{R})$. We write $N \subseteq M$\ if $N$ is a submodule of $M$, $N  \subseteq ^{e s s}M$ if $N$ is an essential submodule of $M$, $N  \subseteq ^{ \oplus }M$ if $N$ is a direct summand of $M$, and $N \ll M$ if $N$ is a small submodule of $M$. We will write $N^{k}$ for a direct sum of $k$ copies of $N$, and $U (R)$ for the set of units of $R$. We use $r_{R} (x)$ to denote the right annihilator of $x$ in $R$, and $l_{R} (x)$ the left annihilator of $x$ in $R$. 

\bigskip 

\section{Hypersimple Rings}
\begin{definition}
A ring $R$ is called left (resp. right)  {hypersimple} if the injective hull of any simple left (resp. right) $R$-module is a cyclic $R$-module. 
\end{definition}

While Caldwell's hypercyclic rings demand that the injective hulls of cyclic modules (in
particular of the module $R_{R}$) are cyclic, the restriction to simple modules seems to be a real weakening.

\begin{example}
If $R$ is right Kasch (i.e. every simple right $R$-module is embedded in $R_{R}$) and $E (R_{R})$ is cyclic, then clearly $R$ is right hypersimple. 
\end{example}

\begin{lemma}
\label{projective_implies_cyclic} Let $R$ be any ring and $S$ a simple right $R$-module. If $E (S)$ is projective, then it is also cyclic and in particular isomorphic to a right ideal of $R$ generated by an idempotent. 
\end{lemma}

\begin{proof}
Since $E (S)$ is projective, there exists an embedding $\sigma  :E (S) \rightarrow F = \oplus _{i \in I}R$ of $E (S)$ into a free module. Now, there is a projection map $\pi  :F \rightarrow R$ such that the composition $\pi  \circ \sigma _{ \mid _{S}} :S \rightarrow R$ is non-zero, and hence a monomorphism. But then $\pi  \circ \sigma  :E (S) \rightarrow R$ is non-zero, and hence a monomorphism since $S  \subseteq ^{e s s}E (S)$. This shows that $E (S)$ embeds into $R$ and so it is isomorphic to a direct summand of $R_{R}$.That is, $E (S) \cong e R$ for some $e^{2} =e \in R$. 
\end{proof}

\begin{example}\label{exa:cogenerator}
\label{example_cogenerator} A ring $R$ is a right cogenerator ring if and only if all injective hulls of simple right $R$-modules are projective. It follows from Lemma \ref{projective_implies_cyclic} that right cogenerator rings are right hypersimple. 
\end{example}

\begin{example}
A ring $R$ is called right pseudo-Frobenius (right $P F$-ring) if $R$ is an injective cogenerator in Mod-$R$; equivalently if $R$ is a semiperfect right self-injective ring with essential right socle. By Example \ref{exa:cogenerator}, right $P F$-rings are right hypersimple. 
\end{example}

\begin{theorem}
\label{Pseudo-Frobenius Rings in terms of E(S) being projective}
The following conditions on a ring $R$ are equivalent:

\begin{enumerate}
\item $R$ is right pseudo-Frobenius.

\item $R$ is semilocal and $E (S)$ is projective, for all simple right $R$-modules $S$.

\item $E (S)$ is projective whenever $S$ is a simple right $R$-module and there are only finitely many isomorphism classes of simple right $R$-modules. \end{enumerate}
\end{theorem}

\begin{proof}
$(1) \Rightarrow (2)$. Let $S$ be a simple right $R$-module. Since $R$ is right Kasch, $S$ embeds in $R$, and so $E (S)$ is a summand of $R$, and hence projective. 

$(2) \Rightarrow (3)$. The implication is obvious. 

$(3) \Rightarrow (1)$. By Example \ref{example_cogenerator}, $R$ is a right cogenerator ring. By \cite[Theorem 19.25]{Lam},
$R$ is right pseudo-Frobenius. 
\end{proof}

\begin{example}
\cite[Example 2]{O3}\label{Osofsky's Example}
Let $R$ be an algebra over a field $F$ with basis $\left \{1\right \} \cup \left \{e_{i} \mid i \geq 0\right \} \cup \left \{x_{i} \mid i \geq 0\right \}$ such that $1$ is the unity of $R$ and, for all i and j, $e_{i} e_{j} =\delta _{i j} e_{j}\text{,}$ $x_{j} e_{i} =\delta _{i ,j -1} x_{j}\text{,}$ $e_{i} x_{j} =\delta _{i j} x_{j}\text{,}$ and $x_{i} x_{j} =0\text{,}$ where $\delta _{i j}$ is the Kronecker delta. As shown in \cite{O3}, $R$ is a non-injective right cogenerator ring that is neither left Kasch nor semilocal. Each $e_{i} R$ is injective and $x_{i} R =s o c (e_{i} R)  \subseteq ^{e s s}e_{i} R\text{,}$ for all $i \geq 0.$ Clearly, if $I$ is a minimal right ideal of $R$ then $I =x_{i} R$ for some $i \geq 0\text{,}$ and $x_{i} R \cong x_{j} R$ if and only if $i =j$. By $(2)$ of Theorem \ref{Pseudo-Frobenius Rings in terms of E(S) being projective},
$R$ is a right hypersimple ring, that is not right self-injective. Moreover, since $R$ is not semilocal, it follows from \cite[Lemma 4.1]{NY},
that $R$ is not left self-injective.\ In particular, $R$ is neither left nor right hypercyclic. 
\end{example}

\begin{example}
A ring $R$ is called a $V$-ring if every simple left (resp. right) $R$-module is injective. For such a ring $R$, any simple left $R$-module $S$ coincides with its injective hull $E (S) =S$ and is therefore cyclic. Thus $V$-rings provide a large class of hypersimple rings. It is well-known, that a commutative ring is a $V$-ring if and only if it is von Neumann regular (see \cite{MV}).
There exist also non-commutative Noetherian $V$-rings that are domains (see \cite{CF}). Moreover, it
was shown in \cite[Corollary 2.7]{IY 5}
that if $R$ is a commutative ring such that $E (R)$ is cyclic then $R$ is self-injective. This shows that every commutative regular ring that is not self-injecive is an example of a hypersimple ring that
is not hypercyclic. 
\end{example}

\begin{example}
A submodule $N$\ of a right $R$-module $M$ is said to lie over a direct summand of $M$ if, there is a decomposition $M =M_{1} \oplus M_{2}$ with $M_{1} \subseteq N$ and $N \cap M_{2} \ll M$. A module $M$ is called lifting if every submodule $N$ of $M$ lies over a direct summand of $M$. A ring $R$ is called right Harada ($H$-ring for short) \cite[Theorem 3.1.12]{BO}
if every injective right $R$-module is lifting. $H$-rings are known to be two-sided Artinian, see for example \cite{BO}.
Every right $H$-ring is right hypersimple. For, if $S$ is a simple right $R$-module then $E (S)$ is an indecomposable and lifting module. Consequently, every proper submodule of $E (S)$ is small. Since $R$ is Artinian, $E (S)$ has a maximal submodule, say $N$. Now, if $x \in E (S)$ such that $x \notin N$, then $E (S) =x R +N$, and so $E (S) =x R$. This shows that every right $H$-ring that is not quasi-Frobenius is an example of a right hypersimple ring that is not hypercyclic. 
\end{example}

A ring $R$ is called right uniserial (right chain) if the right ideals of $R$ are totally ordered by inclusion (i.e. if $A$ and $B$ are right ideals of $R$ then $A \subseteq B$ or $B \subseteq A$).

\begin{example}
\label{Example of a 2-sided art ring that is left SF but not rt GSF}
Let $K$ be a field and let $a \mapsto \bar{a}$ be an isomorphism $K \rightarrow \bar{K} \subseteq K$ where the subfield $\bar{K} \neq K\text{.}$ Let $R$ be the left vector space over $K$ with basis $\{1 ,t\}$ satisfying $t^{2} =0$ and $t a =\bar{a} t$ for all $a \in K$. Then $R$ is a local ring with $R/J \cong K$ and $J^{2} =0.$ Clearly, $J (R) =s o c (R) =R t =K t$ is the only non-trivial left ideal of $R$, in particular $R$ is a local left uniserial and left Artinian ring. Moreover, the ring $R$ is right Artinian if and only if $d i m (_{\bar{K}} K)$ is finite. If $d i m (_{\bar{K}} K) =n \geqslant 2$, then $J (R) =s o c (R)$ as a right ideal is a direct sum of $n$ copies of the same simple module, and if $d i m (_{\bar{K}} K) =\infty $, then $J (R) =s o c (R)$ is the direct sum of infinitely many copies of the same simple module. Furthermore, it is not difficult to see that the following statements
are equivalent:

\begin{enumerate}
\item $E (s o c (R))$ is cyclic as a left $R$-module. 

\item $E (R)$ is cyclic as a left $R$-module. 

\item $R$ is left hypersimple. 

\item $R$ is left self-injective. 

\item $R$ is quasi-Frobenius. 

\item $d i m (_{\bar{K}} K) =1$. \end{enumerate}

This says that if $d i m (_{\bar{K}} K) =n \geq 2$, then $R$ is a local left uniserial and two-sided Artinian ring that is not left hypersimple (hypercyclic). 
\end{example}

Over
commutative rings one is tempted to use local-global arguments to reduce the study of hypersimple rings to local hypersimple rings. For example, Rosenberg
and Zelinsky showed in \cite[Theorem 5]{RZ}
that the injective hulls of simple modules over a commutative ring $R$ have finite length if and only if $R_{\mathcal{M}}$ is Artinian for any maximal ideal $\mathcal{M}$ of $R$, where $R_{\mathcal{M}}$ is the localization of $R$ at $\mathcal{M}$. Vamos showed in \cite[Theorem 2]{V}
that the injective hulls of simple modules over a commutative ring $R$ are Artinian if and only if $R_{\mathcal{M}}$ is Noetherian for any maximal ideal $\mathcal{M}$ of $R$. Carvalho et al. showed in \cite[Theorem 3.3]{CLS}
that any finitely generated submodule of an injective hull of a simple module over a commutative ring $R$ is Artinian if and only if the same is true for any localization $R_{\mathcal{M}}$ by a maximal ideal $\mathcal{M}$ of $R$. It would be interesting to see whether an analogous local-global result holds for commutative hypersimple rings. The following lemma
shows that one implication of this equivalence holds.

\begin{lemma}
Any localization $R_{\mathcal{M}}$ of a commutative hypersimple ring $R$ by a maximal ideal $\mathcal{M}$ is hypersimple. 
\end{lemma}

\begin{proof}
Let $R$ be a commutative ring, $\mathcal{M}$ a maximal ideal of $R$ and $R_{\mathcal{M}}$ the localization of $R$ by $\mathcal{M}$. By \cite[Proposition
5.6]{SV}, the injective hull $E_{R} (R/\mathcal{M})$ of the simple $R$-module $R/\mathcal{M}$ is also the injective hull of the unique simple $R_{\mathcal{M}}$-module $R_{\mathcal{M}}/\mathcal{M} R_{\mathcal{M}}$. Thus, if $R$ is hypersimple, then $E_{R} (R/\mathcal{M})$ is cyclic as an $R$-module and therefore $E_{R_{\mathcal{M}}} (R_{\mathcal{M}}/\mathcal{M} R_{\mathcal{M}}) =E (R/\mathcal{M})$ is cyclic as $R_{\mathcal{M}}$-module. Hence $R_{\mathcal{M}}$ is hypersimple. 
\end{proof}

%
%
%
%

In \cite[Proposition 1.6]{O1},
Osofsky proved that every semilocal hypercyclic ring has an essential right socle. Unfortunately, there is a gap in the proof of that result, and by slightly
modifying the hypotheses, the same argument can be used to establish Osofsky's result. The problem in the proof of \cite[Proposition 1.6]{O1}
lies in the following argument: if $R$ is semilocal and $E =E (S_{1}) \oplus \cdots  \oplus E (S_{n})$ is the direct sum of the injective hulls of all pairwise non-isomorphic simple modules $S_{i}$, then $E$ is faithful. If the annihilators $L_{i} =r_{R} (E (S_{i})/S_{i}) =\{t \in R :E (S_{i}) t \subseteq S_{i}\}$ were essential right ideals of $R$, for all $i$, then also $L =L_{1} \cap \cdots  \cap L_{n}$ would be essential and since $E L J (R) =0$ and $E$ is faithful, one concludes that $L J (R) =0$, hence $L =s o c (R_{R})$ is essential. The problem is that $L_{i}$ does not need to be essential. If $E (S_{i}) =x_{i} R$ is cyclic, then $L_{i} =r_{R} (x +S_{i}) =\{t \in R :x_{i} t \in S_{i}\} \subseteq ^{e s s}R$ provided $R$ is commutative. However without commutativity, $L_{i}$ might be different from $r_{R} (x +S_{i})$. To remedy the lack of commutativity we introduce a new definition, but first, recall that a submodule $N$ of a module $M$ is fully invariant if every endomorphism of $M$ maps $N$ into $N$. A fully invariant submodule of a module corresponds to a two-sided ideal in a ring. A module $M$ is called duo if every submodule of $M$ is fully invariant, and a ring $R$ is called right duo if $R_{R}$ is a duo module (equivalently if every right ideal of $R$ is an ideal).

\begin{definition}
A right $R$-module $M$ is called essentially duo if every essential submodule of $M$ is fully invariant, and a ring $R$ is called essentially right duo if the module $R_{R}$ is essentially duo (equivalently if every essential right ideal of $R$ is an ideal). 
\end{definition}

The following example show that the class of essentially right duo rings and the
class of right quasi-duo rings, i.e. those whose maximal right ideal is two-sided, are not contained in each other.

\begin{example}
\label{Example of essentially duo that is not quasi-duo}
Any semisimple right $R$-module $M$ is trivially essentially duo, because $M$ has no essential submodules. In particular, any semisimple Artinian ring
\begin{equation*}R =M_{n_{1}} (D_{1}) \times \cdots  \times M_{n_{k}} (D_{k})\text{,}
\end{equation*}for division rings $D_{i}$ and numbers $n_{i} \geq 1$, is left and right essentially duo ring. However, if one of the numbers $n_{i}$ is bigger than $1$, then $R$ contains a right (respectively left) maximal ideal that is not an ideal. Thus $R$ is neither right nor left quasi-duo. Note that $R$ is not even Abelian, where $R$ is called Abelian if its idempotents are central. This shows that there are essentially right duo rings that are not right quasi-duo.

On the contrary, let $R$ be any commutative integral domain that is not a field. Then the ring of lower $2 \times 2$-matrices over $R$, i.e.
\begin{equation*}S =\mathrm{L} \mathrm{T} \mathrm{M}_{2 \times 2} (R) =\left [\begin{array}{cc}R & 0 \\
R & R\end{array}\right ] =\left \{\left [\begin{array}{cc}a & 0 \\
b & c\end{array}\right ] :a ,b ,c \in R\right \}
\end{equation*}is right and left quasi-duo, by \cite[Proposition 2.1]{Y}.
However, if $A$ is any proper non-trivial ideal of $R$, then
\begin{equation*}I =\left [\begin{array}{cc}R & 0 \\
A & A\end{array}\right ] =\left \{\left [\begin{array}{cc}a & 0 \\
b & c\end{array}\right ] :a \in R ,
\:
b ,c \in A\right \}
\end{equation*}is an essential right ideal of $S$ that is not an ideal of $S$, because for any non-zero element $m =\left [\begin{array}{cc}x & 0 \\
y & z\end{array}\right ] \in S$ and non-zero $a \in A$:
\begin{equation*}0 \neq m a =\left [\begin{array}{cc}x & 0 \\
y & z\end{array}\right ] \left [\begin{array}{cc}a & 0 \\
0 & a\end{array}\right ] =\left [\begin{array}{cc}x a & 0 \\
y a & z a\end{array}\right ] \in I\text{,}
\end{equation*}since $R$ is a domain, $a \neq 0$ and one of the components of $m$ is non-zero. Hence $m S \cap I \neq 0$, i.e. $I$ is an essential right ideal of $S$. On the other hand,
\begin{equation*}\left [\begin{array}{cc}0 & 0 \\
1 & 0\end{array}\right ] \left [\begin{array}{cc}R & 0 \\
A & A\end{array}\right ] =\left [\begin{array}{cc}0 & 0 \\
R & 0\end{array}\right ]
\not 
 \subseteq I\text{,}
\end{equation*}i.e. $R I
\not 
 \subseteq I$, which shows that $I$ is not a left ideal of $S$. Thus $S$ is not an essentially right duo ring. This shows that there exist right quasi-duo rings that are not essentially right duo rings. 
\end{example}

\begin{proposition}
\label{If R is semilocal and E(R/J(R)) is cyclic, then R has essential socle}
If $R$ is semilocal, right hypersimple, and essentially right duo, then $s o c (R_{R})  \subseteq ^{e s s}R_{R}$. 
\end{proposition}

\begin{proof}
As $R$ is semilocal, there exists only finitely many non-isomorphic simple right $R$-modules, say $S_{1} ,\ldots  ,S_{n}$. Let $E =E (S_{1}) \oplus \cdots  \oplus E (S_{n})$ be the direct sum of injective hulls of the simples $S_{i}$. By assumption $E (S_{i}) =x_{i} R$ is cyclic, for each $i$. Since $L_{i} =r_{R} (x_{i} +S_{i}) =\{t \in R :x_{i} t \in S_{i}\}$ is an essential right ideal of $R$ and by assumption $L_{i}$ a two-sided ideal, we have $E (S_{i}) L_{i} =x_{i} R L_{i} =x_{i} L_{i} \subseteq S_{i}$. Thus $L_{i} =r_{R} (E (S_{i})/S_{i})$ is essential in $R$ and so is $L =L_{1} \cap \cdots  \cap L_{n}$. Therefore $E L J (R) \subseteq (S_{1} \oplus \cdots  \oplus S_{n}) J (R) =0$ and as $E$ is faithful (since it is a cogenerator), $L J (R) =0$. As $R$ is semilocal, $L \subseteq s o c (R_{R})$, showing $s o c (R_{R})  \subseteq ^{e s s}R_{R}$. 
\end{proof}

\begin{lemma}
\label{If R is semilocal and E(R) is cyclic, then R is semiperfect self-inj}
\cite[Theorem 2.16]{IY 5}
If $R$ is a semilocal ring and $E (R_{R})$ is cyclic, then $R$ is semiperfect and right self-injective. 
\end{lemma}

In \cite{JMS},
a ring $R$ is called right $q$-ring if every right ideal of $R$ is quasi-injective; equivalently if $R$ is right self-injective and essentially right duo. The next result follows directly from Lemma \ref{If R is semilocal and E(R) is cyclic, then R is semiperfect self-inj}, Proposition \ref{If R is semilocal and E(R/J(R)) is cyclic, then R has essential socle}, and \cite[Theorem 3.24 and Example 3.26]{NY}.

\begin{theorem}
\label{If R is semilocal ess rt duo with E(R) and E(R/J(R)) are cyclic, then R is rt PF-ring}If
a ring $R$ is semilocal, right hypersimple, and essentially right duo such that $E (R_{R})$ is cyclic, then $R$ is right pseudo-Frobenius. In particular, $R$ satisfies the following conditions:

\begin{enumerate}
\item $J (R) =Z (R_{R}) =Z (_{R} R)$. 

\item $s o c (R):=s o c (R_{R}) =s o c (_{R} R)\text{.}$ 

\item $s o c (R)  \subseteq ^{e s s}R_{R}$ and $s o c (R)  \subseteq ^{e s s}\text{}_{R} R\text{.}$ 

\item $s o c (e R)$ and $s o c (R e)$ are simple, $s o c (e R)  \subseteq ^{e s s}e R$ and $s o c (R e) \subseteq ^{e s s}\text{}R e$, for every local idempotent $e$ of $R$. 

\item $R$ has left and right finite Goldie dimension. 

\item Every simple right and simple left
$R$-module is embedded in $R$. 

\item $R$ is a right $q$-ring. 

\item $T =r_{R} l_{R} (T)$, for every right ideal $T$ of $R$. \end{enumerate}
\end{theorem}

Observe that if $R$ is a semilocal ring and $E (R/J (R))$ is a cyclic module, then $E (R/J (R))  = \oplus _{i =1}^{n}E (S_{i})$, where $\left \{S_{i}\right \}_{i =1}^{n}$ is a set of simple right $R$-modules. Therefore $E (S_{i})$ is cyclic, $1 \leq i \leq n$. Now, if $S$ is any simple right $R$-module, then $S \cong S_{i}$, for some $i$, $1 \leq i \leq n$, and so $E (S)$ is cyclic. In particular, $R$ is a right hypersimple ring. Now, the next result is an immediate consequence of Theorem \ref{If R is semilocal ess rt duo with E(R) and E(R/J(R)) are cyclic, then R is rt PF-ring} above.

\begin{corollary}
If $R$ is a semilocal essentially right duo ring such that $E (R_{R})$ and $E (R/J (R))$ are cyclic modules, then $R$ is right pseudo-Frobenius that satisfies the conditions $(1)$ through $(8)$ of Theorem \ref{If R is semilocal ess rt duo with E(R) and E(R/J(R)) are cyclic, then R is rt PF-ring} above. 
\end{corollary}

\section{Rings with Cyclic Injective Hulls}
As pointed out in the introduction there are no known examples of hypercyclic rings that are not self-injective. Moreover, we also don't know
of an example of a non self-injective ring $R$ whose injective hull $E (R_{R})$\ is cyclic. In general, it is not easy to construct $E (R_{R})$ if the ring $R$ is not Artinian. This means that the existence of a non self-injective ring $R$ whose injective hull $E (R_{R})$\ is cyclic will be of an elusive type of a ring. 

In this section we
continue the investigation, that was carried out in \cite{IY 5}, of the
rings $R$ whose injective hull $E (R_{R})$ is cyclic. We will extend and unify some of the results in \cite{IY 5}
and obtain new ones.

\begin{lemma}
\label{Basic Lemma I} Denote by $E$ an injective hull of $R_{R}$ with embedding $\epsilon  :R \rightarrow E$. Suppose that $E$ is cyclic. Then there exist a surjective homomorphism $\pi  :R \rightarrow E$, and endomorphisms $f :R \rightarrow R$ and $g :E \rightarrow E$ of right $R$-modules such that $\epsilon  =\pi  \circ f$ and $\pi  =g \circ \epsilon $ hold, i.e. the following diagrams commute:
\begin{equation*}
\xymatrix { & &R\ar @{->}[d]^\epsilon \ar @{-->}[lld]_f\\ R \ar @{->>}[rr]_{\pi } && E }
\text{\quad \quad \quad \quad }
\xymatrix { R \ar @{->}[rr]^{\epsilon }\ar @{->>}[d]_\pi && E \ar @{-->}[lld]^{g}\\ E }
\end{equation*}Furthermore, the following properties hold:

\begin{enumerate}
\item $f$ is injective and if $\mathrm{I} \mathrm{m} (f)$ is essential in $R$, then $R \simeq E$. 

\item $\mathrm{K} \mathrm{e} \mathrm{r} (\pi ) \oplus \mathrm{I} \mathrm{m} (f)$ is an essential right ideal of $R$. 

\item $g$ is surjective and if $\mathrm{K} \mathrm{e} \mathrm{r} (g)$ is small in $E$, then $R \simeq E$. 

\item $\mathrm{I} \mathrm{m} (\epsilon ) +\mathrm{K} \mathrm{e} \mathrm{r} (g) =E$ and $\mathrm{I} \mathrm{m} (\epsilon ) \cap \mathrm{K} \mathrm{e} \mathrm{r} (g) =\epsilon  (\mathrm{K} \mathrm{e} \mathrm{r} (\pi ))$. 

\item $E \simeq E (\mathrm{K} \mathrm{e} \mathrm{r} (\pi )) \oplus E$ and $E/\epsilon  (\mathrm{K} \mathrm{e} \mathrm{r} (\pi )) \simeq E/\epsilon  (R) \oplus E$. \end{enumerate}
\end{lemma}

\begin{proof}
(1) Since $\epsilon  =\pi  \circ f$ is injective, also $f$ is injective. Furthermore, if $\mathrm{I} \mathrm{m} (f)$ is essential, then $\pi $ is injective and hence an isomorphism. 

(2) From $\epsilon  =\pi  \circ f$ and $f$ being injective, we get $\mathrm{K} \mathrm{e} \mathrm{r} (\pi ) \cap \mathrm{I} \mathrm{m} (f) =\mathrm{K} \mathrm{e} \mathrm{r} (\epsilon ) =0$. Furthermore, $\mathrm{K} \mathrm{e} \mathrm{r} (\pi ) \oplus \mathrm{I} \mathrm{m} (f) =\pi ^{ -1} (\mathrm{I} \mathrm{m} (\epsilon ))$ holds, because if $r \in \pi ^{ -1}(\mathrm{I} \mathrm{m} (\epsilon )$, then $\pi  (r) =\epsilon  (r^{ \prime }) =\pi  (f (r^{ \prime }))$ for some $r^{ \prime } \in R$. Therefore $r -f (r^{ \prime }) \in \mathrm{K} \mathrm{e} \mathrm{r} (\pi )$ and $r =r -f (r^{ \prime }) +f (r^{ \prime }) \in \mathrm{K} \mathrm{e} \mathrm{r} (\pi ) \oplus \mathrm{I} \mathrm{m} (f)$. Conversely, if $r \in \mathrm{K} \mathrm{e} \mathrm{r} (\pi ) \oplus \mathrm{I} \mathrm{m} (f)$, then $r =r^{ \prime } +f (r^{ \prime  \prime })$ for some $r^{ \prime } \in \mathrm{K} \mathrm{e} \mathrm{r} (\pi )$ and $r^{ \prime  \prime } \in R$. Hence $\pi  (r) =\pi  (f (r^{ \prime  \prime })) =\epsilon  (r^{ \prime  \prime })$, i.e. $r \in \pi ^{ -1} \left (\mathrm{I} \mathrm{m} (\epsilon )\right )$. 

Since $\mathrm{I} \mathrm{m} (\epsilon )$ is essential in $E$ and since pre-images of essential submodules are essential we conclude that $\mathrm{K} \mathrm{e} \mathrm{r} (\epsilon ) \oplus \mathrm{I} \mathrm{m} (f)$ is essential in $R$. 

(3) and (4) are dual to (1) and (2) with the exception that $\mathrm{K} \mathrm{e} \mathrm{r} (\pi )$ is not necessarily small: (3) Since $\pi  =g \circ \epsilon $ is surjective, also $g$ is injective. Furthermore, if $\mathrm{K} \mathrm{e} \mathrm{r} (g)$ is small, then $\epsilon $ is surjective and hence an isomorphism. 

(4) From $\pi  =g \circ \epsilon $ and $g$ being surjective, we get $\mathrm{I} \mathrm{m} (\epsilon ) +\mathrm{K} \mathrm{e} \mathrm{r} (g) =\mathrm{I} \mathrm{m} (\pi ) =E$, because for any $e \in E$ we have $g (e) =\pi  (r)$ for some $r \in R$. Hence $e -\epsilon  (r) \in \mathrm{K} \mathrm{e} \mathrm{r} (g)$ and $e =\epsilon  (r) +(e -\epsilon  (r)) \in \mathrm{I} \mathrm{m} (\epsilon ) +\mathrm{K} \mathrm{e} \mathrm{r} (g)$. Furthermore, $\mathrm{I} \mathrm{m} (\epsilon ) \cap \mathrm{K} \mathrm{e} \mathrm{r} (g) =\epsilon  (\mathrm{K} \mathrm{e} \mathrm{r} (\pi ))$ holds, because if $r \in \epsilon  (\mathrm{K} \mathrm{e} \mathrm{r} (\pi ))$, then $r =\epsilon  (r^{ \prime })$ for some $r^{ \prime } \in \mathrm{K} \mathrm{e} \mathrm{r} (\pi )$. Hence $g (r) =g (\epsilon  (r^{ \prime })) =\pi  (r^{ \prime }) =0$ shows $r \in \mathrm{I} \mathrm{m} (\epsilon ) \cap \mathrm{K} \mathrm{e} \mathrm{r} (g)$. On the other hand, if $r =\epsilon  (r^{ \prime }) \in \mathrm{I} \mathrm{m} (\epsilon ) \cap \mathrm{K} \mathrm{e} \mathrm{r} (g)$, then $\pi  (r^{ \prime }) =g (r) =0$, i.e. $r \in \epsilon  (\mathrm{K} \mathrm{e} \mathrm{r} (\pi ))$. 

The first part of (5) follows from (2), because $\mathrm{K} \mathrm{e} \mathrm{r} (\pi ) \oplus \mathrm{I} \mathrm{m} (f)$ is essential in $R$:
\begin{equation*}E =E (R_{R}) =E (\mathrm{K} \mathrm{e} \mathrm{r} (\pi )) \oplus E (\mathrm{I} \mathrm{m} (f)) \simeq E (\mathrm{K} \mathrm{e} \mathrm{r} (\pi )) \oplus E
\end{equation*}as $\mathrm{I} \mathrm{m} (f) \simeq R$. The second part of (5) follows from (4), because
\begin{eqnarray*}E/\epsilon  (\mathrm{K} \mathrm{e} \mathrm{r} (\pi )) &\simeq& \mathrm{I} \mathrm{m} (\epsilon )/\left (\mathrm{I} \mathrm{m} (\epsilon ) \cap \mathrm{K} \mathrm{e} \mathrm{r} (g)\right ) \oplus \mathrm{K} \mathrm{e} \mathrm{r} (g)/\left (\mathrm{I} \mathrm{m} (\epsilon ) \cap \mathrm{K} \mathrm{e} \mathrm{r} (g)\right ) \\
 &  \simeq  & E/\mathrm{K} \mathrm{e} \mathrm{r} (g) \oplus E/\mathrm{I} \mathrm{m} (\epsilon ) \\
 &  \simeq  & \mathrm{I} \mathrm{m} (g) \oplus E/\mathrm{I} \mathrm{m} (\epsilon ) \simeq E \oplus E/\epsilon  (R)\text{.}\end{eqnarray*}
\end{proof}

Under the assumptions that $E =E (R_{R})$ is cyclic with surjective homomorphism $\pi  :R \rightarrow E$, Lemma \ref{Basic Lemma I}(5) shows that if $\pi $ is not injective, then $E$ is isomorphic to a proper direct summand of itself. A right $R$-module $M$ is called  {Dedekind-finite} if $M$ is not isomorphic to a proper direct summand of itself. Hence if $E$ is cyclic and Dedekind-finite, then Lemma \ref{Basic Lemma I}(5) implies that $\pi $ is an isomorphism and hence $R \simeq E$ is self-injective. We reprove in this way \cite[Proposition 2.9]{IY 5}.
It is well-known that a ring $R$ is Dedekind-finite as right (resp. left) $R$-module if and only if it is  {directly-finite}, which means that whenever $a b =1$, then $b a =1$, for any $a ,b \in R$. Dedekind-finite modules include Hopfian and co-Hopfian modules. Recall that a right $R$-module $M$ is  {Hopfian} resp.  {generalized Hopfian} (see \cite{GH}),
if any epimorphism of $M$ is an isomorphism resp. has small kernel; while $M$ is called  {co-Hopfian} resp.  {weakly co-Hopfian} (see \cite{HV})
if every monomorphism is an isomorphism resp. has essential image. Examples of generalized Hopfian modules include Noetherian modules and modules with finite dual Goldie dimension, e.g. Artinian modules and finitely generated modules over semilocal rings. In particular, if $R_{R}$ has a cyclic injective hull $E$ such that $E$ is generalized Hopfian (e.g. if $R$ is semilocal), then $R$ is right self-injective by Lemma \ref{Basic Lemma I}(3). 

Examples
of weakly co-Hopfian modules include modules with finite Goldie dimension as well as cyclic modules over commutative rings. Moreover, every square-free
module $M$ is weakly co-Hopfian, where a module $M$ is called  {square-free} \textrm{(}$S F$-module for short\textrm{)} if it contains no non-zero isomorphic submodules $A$ and $B$ with $A \cap B =0$. By \cite[Proposition 1.4]{HV},
the notions Dedekind-finite, co-Hopfian and weakly co-Hopfian coincide for quasi-injective modules. 

We will now compile a list of
sufficient conditions that force a ring to be self-injective under the condition that its injective hull is cyclic. First of all we recover \cite[Propositions 2.4 and 2.9]{IY 5}
in the following Proposition.

\begin{proposition}
\label{If E(R) is cyclic and R is weakly co-Hopf, then R is injective}
\label{If E(R) is cyclic and quasi Dedekind-finite, then R is co-Hopfian}
\label{E(R) cyclic implies Injective IV} \label{E(R) cyclic implies Injective III}
\label{If E(R) is cyclic and Dedekind-finite, then R is injective}
Let $E =E (R_{R})$ be an injective hull of $R$. Then the following conditions are equivalent:

\begin{enumerate}
\item [(a)] $E$ is cyclic and Dedekind-finite. 

\item [(b)] $E$ is cyclic and $R_{R}$ is weakly co-Hopfian. 

\item [(c)] $E$ is cyclic and, whenever $E =x R$, $r_{R} (x) =0$. 

\item [(d)] $R$ is right self-injective and directly finite. \end{enumerate}
\end{proposition}

\begin{proof}
$(a) \Rightarrow (b)$. By \cite[Corollary 1.5]{HV}, $R_{R}$ is weakly co-Hopfian.

$(b) \Rightarrow (d)$. The implication follows by \cite[Proposition 2.4]{IY 5} or also by Lemma \ref{Basic Lemma I}(1).

$(d) \Rightarrow (a)$. The implication follows trivially for $R \simeq E$.

$(d) \Rightarrow (c)$. Clearly, if $R$ is right self-injective, then there exists an isomorphism $\pi  :R \rightarrow E$ of right $R$-modules. If $E =x R$, then $R =x^{ \prime } R$, for $x^{ \prime} = \pi^{ -1} (x)$. So $1 =x^{ \prime } y$ for some $y \in R$. As $R$ is directly finite, $y x^{ \prime } =1$. So, $x^{ \prime } \in U (R)$ and hence $r_{R} (x) =r_{R} (x^{ \prime }) =0$. 

$(c) \Rightarrow (d)$. Clearly, $E =x R \simeq R/r_{R} (x) \simeq R$ shows $R$ is right self-injective. If $x y =1$ in $R$, then $x R =R =E (R_{R})$, so $r_{R} (x) =0$. As $x (1 -y x) =0$, it follows that $y x =1$. So $R$ is directly finite. 
\end{proof}

A well-known example of a right self-injective ring that is not directly finite is
the (full) ring of linear transformations of an infinite dimensional vector space over a division ring, with linear transformations acting on the left of
vectors (see e.g. \cite{O}). The above proposition leads to the following
question:

\begin{problem}
\label{Problem 1} Is $R$ right self-injective if $E (R_{R})$ is cyclic and $R$ is directly finite? 
\end{problem}

\begin{example}
\label{M is DF but E(M) is not DF}Observe that if $E (M)$ is Dedekind-finite then $M$ is Dedekind-finite but the converse in general need not be true. For, if $\mathbb{Z}_{n} =\mathbb{Z}/n \mathbb{Z}$, $p$ is a prime number and $\mathbb{Z}_{p^{\infty }}$ is the Pr{\"u}fer $p$-group, then the $\mathbb{Z}$-module $M = \oplus _{n \geq 1}\mathbb{Z}_{p^{n}}$ is Dedekind-finite, but its injective hull $E (M) = \oplus _{n \geq 1}\mathbb{Z}_{p^{\infty }}$ is not a Dedekind-finite $\mathbb{Z}$-module. Moreover, $M$ has a homomorphic image which is a direct sum of countably infinite copies of $\mathbb{Z}_p$ and which is not Dedekind-finite. 
\end{example}

Recall first, a right $R$-module $M$ is called generalized $C 2$-module ($G C 2$-module for short) if, any submodule $N$ of $M$ that is isomorphic to $M$, is a direct summand of $M$. A ring $R$ is called a right $G C 2$-ring if $R_{R}$ is a $G C 2$-module and it is easy to see that a module $M$ is co-Hopfian if and only if $M$ is Dedekind-finite and $G C 2$. Hence we have a partial positive answer to Problem \ref{Problem 1} above, namely by Proposition
\ref{If E(R) is cyclic and R is weakly co-Hopf, then R is injective}:
if $R$ is directly finite and $E (R_{R})$ is cyclic, then $R$ is right self-injective if and only if $R$ is a right $G C 2$-ring. 

\medskip A module $M$ is called  {strongly Dedekind-finite} if every homomorphic image of $M$ is Dedekind-finite. 
The module $M$ in Example \ref{M is DF but E(M) is not DF} is a Dedekind-finite $\mathbb{Z}$-module that is not strongly Dedekind-finite. 

\begin{proposition}
\label{E(R) cyclic implies Injective V} Suppose $E (R_{R})$ is cyclic. If $R_{R}$ is strongly Dedekind-finite, then $R$ is right self-injective. Moreover, in this case $R/J (R)$ is unit-regular. 
\end{proposition}

\begin{proof}
If $R_{R}$ is strongly Dedekind-finite and $E$ is cyclic, then $E$ is Dedekind-finite and the claim follows again by Proposition \ref{If E(R) is cyclic and R is weakly co-Hopf, then R is injective}. 
\end{proof}

A module $M$ is called  {dual-square-free} \textrm{(}$D S F$-module for short\textrm{)} if $M$ has no proper submodules $A$ and $B$ with $M =A +B$ and $M/A \cong M/B$. A ring $R$ is called right $D S F$-ring, if $R$ as a right $R$-module is a $D S F$-module. By \cite[Proposition 2.3 and Lemma
3.1]{IY 2}, $D S F$-modules are strongly Dedekind-finite. Moreover, it was shown in \cite[Corollary 2.17]{IY 2}
that a ring $R$ is right $D S F$ if and only if $R$ is  {right quasi-duo}, i.e. every maximal right ideal of $R$ is two-sided. Moreover, from \cite{Y} we have that
$R$ is right quasi-duo if and only if $R/J (R)$ is a right quasi-duo ring, if and only if $R/J (R)$ is an Abelian right quasi-duo ring.

\begin{corollary}[{compare with \cite[Proposition 3.8]{IY 5}}]
\label{If E(R) is cyclic and DSF then R is self-inj}
\label{E(R) cyclic implies Injective VI} The following conditions are equivalent
for a ring $R$:

\begin{enumerate}
\item [(a)] $E (R_{R})$ is cyclic and $R$ is right quasi-duo. 

\item [(b)] $E (R_{R})$ is cyclic and $R/J (R)$ is strongly regular. 

\item [(c)] $R$ is right quasi-duo right self-injective. \end{enumerate}
\end{corollary}

\begin{proof}
Since a right quasi-duo ring is strongly Dedekind-finite, the equivalence $(a) \Leftrightarrow (c)$ follows from Proposition \ref{If E(R) is cyclic and Dedekind-finite, then R is injective}.
\end{proof}

A right $R$-module $M$ is called  {distributive} if $A \cap (B +C) =(A \cap B) +(A \cap C)$ for all submodules $A$, $B$, and $C$ of $M$. It was shown in \cite{S} that a module $M$ is distributive if and only if every quotient of $M$ is square-free and, in \cite[Lemma 2.2]{IY 6},
$M$ is distributive if and only if every submodule of $M$ is $D S F$. Hence from Corollary \ref{If E(R) is cyclic and DSF then R is self-inj}
it follows that if $R$ is a right distributive ring and $E (R_{R})$ is cyclic, then $R$ is right self-injective. 

\medskip The next proposition provides conditions under
which a ring is right weakly co-Hopfian. Hence those conditions are sufficient for $R$ to be self-injective in case it has a cyclic injective hull. A ring $R$ is called a  {$U N$-ring} \cite{GC} if, every non-unit of $R$ is a product of a unit and a nilpotent. An element $a$ in a ring $R$ is called right  {morphic} \cite{NS} if $R/a R \cong r_{R} (a)$. An element $a$ in a ring $R$ is (von Neumann)  {regular} (resp.  {unit regular}) if $a \in a R a$ resp. $a \in a U (R) a$. By \cite[Example 4 and Proposition 5]{NS},
an element in a ring is unit regular if and only if it is both regular and right morphic. Recall that a ring $R$ is  {strongly $\pi $-regular} if and only if, for each $a \in R$, there exists $n \geq 1$ such that $a^{n}$ is the product of a unit and an idempotent that commute. Say that a non-zero element $a \in R$ is a left  {non-zero divisor}, if $r_{R} (a) =0$, i.e. if $a b =0$ implies $b =0$. Note that if $a$ is a left non-zero divisor and has a right inverse, say $b$, then $1 =a b$ implies that $a (1 -b a) =0$. But then $1 =b a$ shows that $a \in U (R)$ is invertible.

\begin{proposition}
\label{E(R) cyclic implies Injective II} A ring
$R$ is right weakly co-Hopfian if it satisfies any of the following conditions:

\begin{enumerate}
\item Every non-invertible left non-zero divisor $a \in R$ satisfies $a (b R) \subseteq (b R) a$ for any $b \in R$. 

\item Every non-invertible left non-zero divisor is central. 

\item A power of each left non-zero divisor is morphic. 

\item $R$ is a commutative or morphic or strongly $\pi $-regular or $U N$-ring. \end{enumerate}

If moreover $E (R_{R})$ is cyclic, then $R_{R}$ is injective. 
\end{proposition}

\begin{proof}
(1) Suppose $a \in R$ is a left non-zero divisor of $R$. If $a$ is invertible, then clearly $a R =R$ is essential. Otherwise, $a$ satisfies $a (b R) \subseteq (b R) a$, for any $b \in R$. Let $0 \neq b \in R$ be any non-zero element. Since $r_R(a) =0$ (and $R$ has an identity), $a b R \neq 0$. By hypothesis $0 \neq a b R \subseteq (b R) a \cap a R \subseteq b R \cap a R$. Thus $a R$ is essential in $R$ and hence $R$ is right weakly co-Hopfian. 

(2) Follows from (1). 

(3) Let $a$ be a left non-zero divisor of $R$. By hypothesis, there exists $n \geq 1$ such that $R/a^{n} R \simeq r_{R} (a^{n}) =0$, as $r_{R} (a) =0$. Hence $R =a^{n} R$, which shows that $a$ is invertible. 

(4) The mentioned classes of rings imply one of the above properties. 
\end{proof}

\medskip In \cite{Z}, a submodule $N$ of a module $M$ is $\delta $-small in $M$ if $M \neq X +N$ for any proper submodule $X$ with $M/X$ singular. The sum of all $\delta $-small submodules of a module $M$ is denoted by $\delta  (M)$, and we use $\delta  (R)$ for $\delta  (R_{R})$. As shown in \cite{Z}, $\delta  (R)$ is an ideal of $R$ which is the intersection of all essential maximal right ideals of $R$, so $s o c (R_{R}) \subseteq \delta  (R)$ and $J (R) \subseteq \delta  (R)$; indeed, $J (R/s o c (R_{R})) =\delta  (R)/s o c (R_{R})$.

\begin{definition}
\cite{EKP} A module $M$ is called  {$\delta $-Hopfian} if every surjective endomorphism of $M$ has a $\delta $-small kernel in $M$. 
\end{definition}

In our discussion after Lemma \ref{Basic Lemma I}, we
have already mentioned that if $E (R_{R})$ is cyclic and generalized Hopfian, then $R_{R}$ is injective. Generalized Hopfian modules are $\delta $-Hopfian and, by \cite[Example 1]{EKP},
the converse need not be true.

\begin{proposition}
\label{If E(R) is cyclic and delta-Hopfian, then R is injective}
If $E (R_{R})$ is cyclic and $\delta $-Hopfian, then $R_{R}$ is injective. 
\end{proposition}

\begin{proof}
Let $E =E (R_{R})$ be cyclic with embedding $\epsilon  :R \rightarrow E$ and canonical projection $\pi  :R \rightarrow E$. There exists an epimorphism $g :E \rightarrow E$ such that $E =\mathrm{I} \mathrm{m} (\epsilon ) +\mathrm{K} \mathrm{e} \mathrm{r} (g)$, by Lemma \ref{Basic Lemma I}. Since $E$ is $\delta $-Hopfian, $\mathrm{K} \mathrm{e} \mathrm{r} (g)$ is $\delta $-small in $E$. Now, by \cite[Lemma 1.2]{Z},
it follows that $E =\mathrm{I} \mathrm{m} (\epsilon ) \oplus Y$ where $Y$ is a projective semisimple submodule of $E$. Inasmuch as $\mathrm{I} \mathrm{m} (\epsilon )$ is essential in $E$, $Y =0$ and $\epsilon $ is an isomorphism. 
\end{proof}

\begin{lemma}
\label{Basic Lemma II}Let $E (R_{R}) =x R$ and write $x_{0}:=1_{R} =x a$ with $a \in R$. If, for some $u ,v \in U (R)$ and some $n \geq 1$, $(u a v)^{n} \in \delta  (R)$, then $R$ is a semisimple ring. 
\end{lemma}

\begin{proof}
Let $y_{0}:=x_{0} v =(x u^{ -1}) (u a v) =y b$, where $y =x u^{ -1}$ and $b =u a v$. Then $E (R_{R}) =y R$. Let $f :y_{0} R \rightarrow y R$ be given by $f (y_{0} r) =y r$. As $r_{R} (y_{0}) =0$, $f$ is well-defined. Since $(y R)_{R}$ is injective, there exists an $R$-homomorphism $g :y R \rightarrow y R$ such that $f =g$ on $y_{0} R$. Then
\begin{equation*}y =f (y_{0}) =g (y_{0}) =g (y b) =g (y) b\text{.}
\end{equation*}If we write $g (y) =y c$ where $c \in R$, we get
\begin{equation*}y =g (y) b =y c b\text{.}
\end{equation*}Thus, $g (y) =g (y c b) =g (y) c b =y c^{2} b$, so
\begin{equation*}y =g (y) b =y c^{2} b^{2}\text{.}
\end{equation*}Arguing as above repeatedly, one has
\begin{equation*}y =y c^{n} b^{n}\text{\thinspace \thinspace \thinspace }\text{for all}\text{\thinspace \thinspace \thinspace }n \geq 1.
\end{equation*}If $b^{n} \in \delta  (R)$, then $y \in y \delta  (R) \subseteq \delta  (y R)$, so $y R =\delta  (y R)$. Since $\delta  (y R)$ is a $\delta $-small submodule of $y R$, it follows that $y R$ is projective and semisimple by \cite[Lemma
1.2]{Z}. Since $R_{R}$ is essential in $y R$, it follows that $R =y R$ is semisimple. 
\end{proof}

\begin{proposition}
\label{E(R) cyclic and R/J(R) is a UN-ring implies Injective}
\label{E(R) cyclic implies Injective VII} Suppose $R_{R}$ has a cyclic injective hull. Then $R_{R}$ is injective under any of the following conditions:

\begin{enumerate}
\item $R/J (R)$ is a $U N$-ring. 

\item $R/J (R)$ is unit-regular and $J (R)$ is nil. 

\item $R/\delta  (R)$ is semisimple. \end{enumerate}
\end{proposition}

\begin{proof}
(1) Let $\overline{R} =R/J (R)$. Write $E (R_{R}) =x R$ and $x_{0}:=1_{R} =x a$ with $a \in R$. If $a \in U (R)$, then $x =x_{0} a^{ -1} \in R$, so $R_{R} =x R$ is injective. If $a \notin U (R)$, then $\overline{a} \notin U (\overline{R})$, so $\overline{a} =\overline{u} \overline{b}$ where $\overline{u} \in U (\overline{R})$ with inverse $\overline{v}$ and $\overline{b}$ is nilpotent. So $\overline{a} \overline{v} =\overline{u} \overline{b} \overline{v}$ is nilpotent, i.e., $(a v)^{n} \in \delta  (R)$ for some $n \geq 1$. As $v \in U (R)$, $R$ is semisimple by Lemma \ref{Basic Lemma II}. 

(2) Write $E (R_{R}) =x R$ and $x_{0}:=1_{R} =x a$ with $a \in R$. As $\overline{R}$ is unit-regular, $\overline{a} =\overline{u}\, \overline{e}$ where $\overline{u} \in U (\overline{R})$ and $\overline{e}^{2} =\overline{e}$, so $u \in U (R)$. Since $E (R_{R}) =(x u) R$ and $x_{0} =x a =(x u) (u^{ -1} a)$ and $\overline{u^{ -1} a} =\overline{e}$, we may assume that $\overline{a}^{2} =\overline{a}$. Hence $a^{2} -a \in J (R)$. Since $J (R)$ is nil, $a^{n} (1 -a)^{n} =(a -a^{2})^{n} =0$ for some $n >0$. Since $r_{R} (a) =0$, it follows that $(1 -a)^{n} =0$, so $a \in U (R)$. Thus, $x =x_{0} a^{ -1} \in R$, so $R =E (R_{R})$ is injective. 

(3) If $R/\delta  (R)$ is semisimple, then by \cite[Proposition 4]{EKP}
every finitely generated $R$-module is $\delta $-Hopfian, so $E (R_{R})$ is $\delta $-Hopfian. Hence $R_{R}$ is injective by Proposition \ref{If E(R) is cyclic and delta-Hopfian, then R is injective}.
\end{proof}

\begin{problem}
\label{Problem 2} Is $R_{R}$ injective if $R_{R}$ has a cyclic injective hull and $R/J (R)$ is unit-regular? 
\end{problem}

Observe that Problem \ref{Problem 2} will have
a positive answer if Problem \ref{Problem 1} has a positive answer. 

\medskip In \cite{DIYZ 2}, a module $M$ is called a $C 4$-module if, whenever $X$ and $Y$ are submodules of $M$ with $M =X \oplus Y$ and $\alpha  :X \rightarrow Y$ is a homomorphism with $k e r (\alpha )  \subseteq ^{ \oplus }X$, we have $\mathrm{I} \mathrm{m} (\alpha )  \subseteq ^{ \oplus }Y$. The class of $C 4$-modules is a simultaneous extension of both the $C 3$-modules and the $S F$-modules.

\begin{proposition}
\label{E(R) cyclic implies Injective VIII}
A ring $R$ is right self-injective if and only if $E (R_{R})$ is cyclic and projective, and $R_{R}$ is a $C 4$-module. 
\end{proposition}

\begin{proof}
The necessity is obvious. For the sufficiency, since $E:=E (R_{R}) =x R$ is cyclic, let $\eta  :R \rightarrow E$ be the natural $R$-epimorhism given by $\eta  (r) =x r$, $r \in R$. Inasmuch as $E$ is projective, there exists an $R$-homomorphism $h :E \rightarrow R_{R}$ such that $\eta  h =i d_{E}$. Clearly, $h$ is monic and $R =\ensuremath{\operatorname*{Im}}h \oplus \ker  \eta $. If $T =h (\ker \eta )$, then $T \cong \ker \eta $, as $h$ is monic. Now, since $T \subseteq \ensuremath{\operatorname*{Im}}h$, $T \cap \ker  \eta  =0$, and since $\ker \eta   \subseteq ^{ \oplus }R_{R}$ and $R_{R}$\ is a $C 4$-module, we infer from \cite[Theorem 2.2]{DIYZ 2}
that $T  \subseteq ^{ \oplus }R_{R}$. Inasmuch as $T \subseteq \ensuremath{\operatorname*{Im}}h \subseteq R$, it follows that $T  \subseteq ^{ \oplus }\ensuremath{\operatorname*{Im}}h$. Since $\ensuremath{\operatorname*{Im}}h$ is injective, $T \cong \ker \eta $ is injective, and so $R =\ensuremath{\operatorname*{Im}}h \oplus \ker  \eta $ is injective. 
\end{proof}

\section{On Hypercyclic Rings}

%

In this section we revisit some of Osofosky's
results in \cite{O1} and highlight some new information about the class
of hypercyclic rings and show that some of the main results in \cite{O1}
can be obtained by requiring only $E (R_{R})$ and/or $E (R/J (R))$ to be cyclic. 

\medskip In \cite{C},
Caldwell proved that a left perfect, right hypercyclic ring is artinian. We will show below if $R$ is a basic right (or left) perfect ring such that $E (R/J^{2})$ is cyclic as a right $R$-module, then $R$ is right Artinian. Moreover, if in addition $E (R_{R})$ is cyclic as a right $R$-module, then $R$ is quasi-Frobenius. Recall first that a semiperfect ring $R$ is called basic if $1$ is a sum of orthogonal non-isomorphic local idempotents of $R$.

\begin{theorem}
Let $R$ be a basic right (or left) perfect ring. If $E (R/J^{2})$ is cyclic as a right $R$-module, then $R$ is right Artinian. If in addition $E (R_{R})$ is cyclic as a right $R$-module, then $R$ is quasi-Frobenius. 
\end{theorem}

\begin{proof}
Since $R$ is a basic semiperfect ring, it follows from \cite[Corollary 2.4]{IY 6}
that $R_{R}$ is a right $D S F$-module (i.e. $R$ is right quasi-duo). Since homomorphic images of $D S F$-modules are again $D S F$, we infer that $E (R/J^{2})$ is $D S F$ as a right $R$-module. We claim that $M:=E (R/J^{2})$ is an $S F$-module. To see this, let $A$ and $B$ be submodules of $M$ with $A\overset{f}{ \cong }B$ and $A \cap B =0.$ We need to show that $A =B =0$. Since $M$ is injective, there exist two submodules $C$ and $D$ of $M$ such that $A  \subseteq ^{e s s}C  \subseteq ^{ \oplus }M$ and $B  \subseteq ^{e s s}D  \subseteq ^{ \oplus }M\text{.}$ Clearly, $C \cap D =0$, $C \oplus D  \subseteq ^{ \oplus }M$, and $C \oplus D$ is injective. Now, the monomorphism $A\overset{f}{ \longrightarrow }B \subseteq D$ extends to a homomorphism $g :C \longrightarrow D\text{.}$ Since $A  \subseteq ^{e s s}C$ $ and $ $B  \subseteq ^{e s s}D\text{,}$ we have $\ker g =0$ and $\ensuremath{\operatorname*{Im}}g  \subseteq ^{e s s}D$. Inasmuch as $C \oplus D$ is injective, we infer that $\ensuremath{\operatorname*{Im}}g  \subseteq ^{ \oplus }D$ and so $\ensuremath{\operatorname*{Im}}g =D$. Therefore $C \cong D$. Now, write $M =C \oplus D \oplus T$ for a submodule $T$ of $M$. Since $M$ is a $D S F$-module, we infer that $C =D =0$, and so $A =B =0$, as required. Since the class of $S F$-modules is closed under submodules, it follows that $(J/J^{2})_{R}$ is an $S F$-module. Since $R$ is semilocal, $(J/J^{2})_{R}$ is a finite direct sum of non-isomorphic simple modules, and hence finitely generated. Now, by \cite[Lemma 11]{O3},
$R$ is right Artinian, proving the first assertion. Since $R$ is right quasi-duo and $E (R_{R})$ is cyclic as a right $R$-module, it follows from Corollary \ref{If E(R) is cyclic and DSF then R is self-inj}
that $R$ is right self-injective. Now the second assertion follows from the fact that a ring $R$ is quasi-Frobenius if and only if $R$ is right Artinian and right (or left) self-injective. 
\end{proof}

\begin{theorem}
\label{Quasi-duo hypercyclic rings}If $R$ is a right hypercyclic ring, then the following are equivalent:

\begin{enumerate}
\item $R$ is a right quasi-duo ring. 

\item The injective hull of every cyclic right $R$-module is a $D S F$-module. 

\item $R$ is a right self-injective, right distributive ring. \end{enumerate}
\end{theorem}

\begin{proof}
$(1) \Rightarrow (2)$. If $T$ is a right ideal of $R$, then $E (R/T)$ is a cyclic right $R$-module and hence a $D S F$-module since $R$ is right quasi-duo. 

$(2) \Rightarrow (3)$. If $A$ is a right ideal of $R$, then $E:=E (R/A)$ is cyclic and hence a $D S F$-module by the hypotheses. We claim that $E$ is a square-free module. To see this, let $A$ and $B$ are submodules of $E$ with $A\overset{f}{ \cong }B$ and $A \cap B =0$. Since $E$ is injective, $A  \subseteq ^{e s s}A_{1}  \subseteq ^{ \oplus }E$ and $B  \subseteq ^{e s s}B_{1}  \subseteq ^{ \oplus }E$, where $A_{1}$ and $B_{1}$ are submodules of $E$. Clearly, $f$ can be extended to a monomorphism $h :A_{1} \rightarrow B_{1}$ with $\ensuremath{\operatorname*{Im}}h  \subseteq ^{e s s}B_{1}$. But since $A_{1}$ is injective, $\ensuremath{\operatorname*{Im}}h  \subseteq ^{ \oplus }B_{1}$, and so $A_{1} \cong B_{1}$. Now, write $E =A_{1} \oplus B_{1} \oplus C$, for a submodule $C \subseteq E$. Since the class of $D S F$-modules is closed under direct summands, $A_{1} \oplus B_{1}$ is a $D S F$-module, a contradiction. This shows that $A =B =0$, proving the claim. Therefore, both $R/A$ and $E (R/A)$ are square-free modules, and hence $R$ is a right distributive ring. Since right distributive rings are right quasi-duo and $E (R_{R})$ is cyclic as a right $R$-module, it follows from Corollary \ref{If E(R) is cyclic and DSF then R is self-inj}
that $R$ is right self-injective. 

$(3) \Rightarrow (1)$. Right distributive rings are right $D S F$, and hence right quasi-duo. 
\end{proof}

\begin{corollary}
\label{Properties of local hypercyclic rings I}
If $R$ is a local right hypercyclic ring, then:

\begin{enumerate}
\item $R$ is right self-injective. 

\item The injective hull of every cyclic right $R$-module is a $D S F$-module. 

\item $R$ is right distributive. 

\item $R$ is right uniform. 

\item $R$ is right uniserial. \end{enumerate}
\end{corollary}

\begin{proof}
$(1)$, $(2)$ \& $(3)$ They follow from Theorem \ref{Quasi-duo hypercyclic rings}, since local rings
are quasi-duo. 

$(4)$ Let $A$ be a right ideal of $R$. Since $R$ is right self-injective by $(1)$, $A$ is essential in a direct summand of $R$. Inasmuch as $R$ is local, $A$ must be essential in $R$. This shows that $R$ is right uniform. 

$(5)$ Let $A$ and $B$ be non-zero right ideals of $R$. By $(4)$, $A \cap B \neq 0$. Now, $(A/A \cap B) \oplus (B/A \cap B) \subseteq R/A \cap B \subseteq E (R/A \cap B) \cong R/T$, where $T$ is a right ideal of $R$. Since $R$ is a local ring, $R/T$ is a local injective module, and hence indecomposable. Therefore, either $A/A \cap B =0$ or $B/A \cap B =0$. This shows that either $A \subseteq B$ or $B \subseteq A$. 
\end{proof}

A ring $R$ is called right principally injective ($p$-injective) \cite{NY1}, if every homomorphism from a
principal right ideal to $R$ is left multiplication by an element of $R$; equivalently if $l_{R} r_{R} (x) =R x$, for all $x \in R$.

\begin{lemma}
\label{Right uniserial right p-injective rings are left uniserial}
Every right uniserial right $p$-injective ring $R$ is left uniserial. 
\end{lemma}

\begin{proof}
Let $x ,y \in R$. Since $R$ is right uniserial, either $r_{R} (x) \subseteq r_{R} (y)$ or $r_{R} (y) \subseteq r_{R} (x)$. Now, since $R$ is right $p$-injective, $l_{R} r_{R} (x) =R x$ and $l_{R} r_{R} (y) =R y$. This shows that either $R x \subseteq R y$ or $R y \subseteq R x$. Now, let $A$ and $B$ be left ideals of $R$ such that $A \nsubseteq B$. Let $x \in A$, $x \notin B$. Now, if $y \in B$, then $R x \nsubseteq R y$. Therefore, $R y \subseteq R x$, and so $B \subseteq R x \subseteq A$ as required. 
\end{proof}

\begin{corollary}
\label{Properties of local hypercyclic rings II}
Every local right hypercyclic ring $R$ is left uniserial. In particular, $R$ is left uniform and left distributive. 
\end{corollary}

\begin{proof}
The claim follows from Corollary \ref{Properties of local hypercyclic rings I}
and Lemma \ref{Right uniserial right p-injective rings are left uniserial}.
The last assertion is clear. 
\end{proof}

Maximal and almost maximal valuation rings play an important role in commutative ring theory.
Uniserial rings are the natural extension of valuation rings in the non-commutative setting. Recall that a ring $R$ is called left maximal if, for any family $\left \{I_{\alpha } \mid \alpha  \in \Lambda \right \}$ of left ideals $I_{\alpha }$\ of $R$, every system of pairwise solvable congruences of the form $x \equiv x_{\alpha }$ $(I_{\alpha })$ with $\alpha  \in \Lambda $ and $x_{\alpha } \in R$ has a simultaneous solution in $R$.  With the help of our results above, it follows immediately from \cite[Theorem 9]{AS} that every local right hypercyclic ring is left maximal.

\begin{theorem}
Every local right hypercyclic ring is left maximal. 
\end{theorem}

Recall that a ring $R$ is called a dual ring \cite{HN} if every right or left
ideal of $R$ is an annihilator.

\begin{proposition}
\label{Local hypercyclic rings with ess socle are dual rings}
Let $R$ be a local right hypercyclic ring, and consider the following conditions:

\begin{enumerate}
\item $R$ is right duo. 

\item $R$ is essentially right duo. 

\item $s o c (R_{R})  \subseteq ^{e s s}R_{R}$. 

\item $R$ is dual. \end{enumerate}

Then $(1) \Leftrightarrow (2) \Rightarrow (3) \Leftrightarrow (4)$. 
\end{proposition}

\begin{proof}
$(1) \Rightarrow (2)$. The implication is obvious. 

$(2) \Rightarrow (1)$. The implication follows from $(3)$ of Corollary \ref{Properties of local hypercyclic rings I}. 

$(2) \Rightarrow (3)$. The implication follows from Proposition \ref{If R is semilocal and E(R/J(R)) is cyclic, then R has essential socle}. 

$(3) \Rightarrow (4)$. By Corollary \ref{Properties of local hypercyclic rings I}, $R$ is right self-injective. Inasmuch as $s o c (R_{R})  \subseteq ^{e s s}R_{R}$, we infer that $R$ is a right $P F$-ring. Consequently, $R$ is an injective cogenerator. This means, if $T$ is a right ideal of $R$ then $R/T$ embeds in a direct product of copies of $R$, and this implies $T =r (X)$, where $X$ is a subset of $R$. Now let $L$ be a left ideal of $R$. By Corollary \ref{Properties of local hypercyclic rings I} and
Corollary \ref{Properties of local hypercyclic rings II}, the left (right) ideals
of $R$ are totally ordered by inclusion. Define $N:=\bigcap _{R x \supseteq L}R x$. Then $N =l_{R} (\sum _{R x \supseteq L}r (R x))$, since $R$ is right $p$-injective. If $L \neq N$, let $y \in N$ and $y \notin L$. Then $R y\varsupsetneqq L$. Now, if $t \in R$ such that $R t \supseteq L$, then by the definition of $N$, it follows that $R t \supseteq N \supseteq R y$. This means that $R y/L$ is a simple left module, and $J y \subseteq L$, where $J =J (R)$. But since $R y/J y$ is also a simple module, we infer that $L =J y$. Now, since $s o c (R_{R})  \subseteq ^{e s s}R_{R}$, we can find an element $z \in R$ such that $0 \neq y z \in s o c (R_{R})$. Since $R$ is a right $P F$-ring, $s o c (R_{R}) =s o c (_{R} R)$, and so $J y z =0$ and $J y \subseteq l_{R} (z R)$. But since $0 \neq R y z$, we infer that $R y\nsubseteqq l_{R} (z R)$. Since the left ideals of $R$ are totally ordered by inclusion, $J y \subseteq l_{R} (z R)\varsubsetneqq R y$. Since $R y/J y$ is a simple module, it follows that $L =J y =l_{R} (z R)$, as required. 

$(4) \Rightarrow (3)$. The implication follows from \cite[Lemma 3.2 and
Theorem 3.5]{HN}. 
\end{proof}

\begin{remark}
We are unable to prove or disprove that $(2) \Leftrightarrow (3)$ in Proposition \ref{Local hypercyclic rings with ess socle are dual rings}
above. 
\end{remark}

The next result is an immediate consequence of Proposition \ref{Local hypercyclic rings with ess socle are dual rings} and \cite[Theorem 5.3]{HN}.

\begin{corollary}
If $R$ is a local right hypercyclic right duo ring, then every finitely generated right (left) $R$-module has finite Goldie dimension. 
\end{corollary}

\section*{Acknowledgments}
Part of this work was carried out while the second author was visiting the Mathematics Department of the University of Porto, Portugal. He would like to express his gratitude to the members of the mathematics department at University of Porto for the financial support and the warm reception. The third author thanks the host for the hospitality received during his visit to Ohio State University when part of the work was carried out there. The first author was  partially supported by CMUP, member of LASI, which is financed by national funds through FCT - Fundação para a Ciência e a Tecnologia, I.P., under the projects with reference UIDB/00144/2020 and UIDP/00144/2020. The research of the second author was supported by the Mathematics Research Institute of the Ohio State University, and that of the third author by a Discovery Grant from NSERC of Canada.

\end{document}